\documentclass[12pt,a4paper, leqno]{article}

\usepackage{fullpage}
\usepackage{amsmath}
\usepackage{amssymb}
\usepackage{amsthm}
\usepackage{mathrsfs}
\usepackage[dvips]{graphicx}
\usepackage{psfrag}
\usepackage[hang,small,bf]{caption}
\newtheorem{theorem}{Theorem}[section]

\newtheorem{proposition}[theorem]{Proposition}
\newtheorem{lemma}[theorem]{Lemma}

\newtheorem{remark}[theorem]{Remark}
\numberwithin{equation}{section}

\begin{document}

\author{David Windisch}\title{\textbf{Logarithmic components of the vacant set for random walk on a discrete torus}}
\author{\normalsize David Windisch \\
\normalsize Departement Mathematik\\
\normalsize ETH Z\"urich\\
\normalsize CH-8092 Z\"urich\\
\normalsize Switzerland \\
\normalsize \texttt{david.windisch@math.ethz.ch}}
\date{}
\maketitle

\thispagestyle{empty}

\begin{abstract}
This work continues the investigation, initiated in a recent work by Benjamini and Sznitman, of percolative properties of the set of points not visited by a random walk on the discrete torus $({\mathbb Z}/N{\mathbb Z})^d$ up to time $uN^d$ in high dimension $d$. If $u>0$ is chosen sufficiently small it has been shown that with overwhelming probability this vacant set contains a unique giant component containing segments of length $c_0 \log N$ for some constant $c_0 > 0$, and this component occupies a non-degenerate fraction of the total volume as $N$ tends to infinity. Within the same setup, we investigate here the complement of the giant component in the vacant set and show that some components consist of segments of logarithmic size. In particular, this shows that the choice of a sufficiently large constant $c_0>0$ is crucial in the definition of the giant component.
\end{abstract}


\vspace{5pt}

\vspace{5pt}
Submitted to EJP on November 15, 2007, final version accepted April 15, 2008.

\newpage

\setcounter{page}{1}

\section{Introduction}

In a recent work by Benjamini and Sznitman \cite{BS07}, the authors consider a simple random walk on the $d$-dimensional integer torus $E=({\mathbb Z}/N{\mathbb Z})^d$ for a sufficiently large dimension $d$ and investigate properties of the set of points in the torus not visited by the walk after $[uN^d]$ steps for a sufficiently small parameter $u>0$ and large $N$. Among other properties of this so-called vacant set, the authors of \cite{BS07} find that for a suitably defined dimension-dependent constant $c_0 > 0$, there is a unique component of the vacant set containing segments of length at least $[c_0 \log N]$ with probability tending to $1$ as $N$ tends to infinity, provided $u>0$ is chosen small enough. This component is referred to as the giant component. It is shown in \cite{BS07} that with overwhelming probability, the giant component is at $|.|_\infty$-distance of at most $N^\beta$ from any point and occupies at least a constant fraction $\gamma$ of the total volume of the torus for arbitrary $\beta, \gamma \in (0,1)$, when $u>0$ is chosen sufficiently small. One of the many natural questions that arise from the study of the giant component is whether there exist also other components in the vacant set containing segments of logarithmic size. In this work, we give an affirmative answer to this question. In particular, we show that for small $u>0$, there exists some component consisting of a single segment of length $[c_1 \log N]$ for a dimension-dependent constant $c_1 >0$ with probability tending to $1$ as $N$ tends to infinity. 

\vspace{12pt}
In order to give a precise statement of this result, we introduce some notation and recall some results of \cite{BS07}. Throughout this article, we denote the $d$-dimensional integer torus of side-length $N$ by 
$$E=({\mathbb Z}/N{\mathbb Z})^d,$$
where the dimension $d \geq d_0$ is a sufficiently large integer (see~(\ref{eq:dim})). $E$ is equipped with the canonical graph structure, where any two vertices at Euclidean distance $1$ are linked by an edge. We write $P$, resp.~$P_x$, for $x \in E$, for the law on $E^{\mathbb N}$ endowed with the product $\sigma$-algebra $\mathcal F$, of the simple random walk on $E$ started with the uniform distribution, resp.~at $x$. We let $(X_n)_{n \geq 0}$ stand for the canonical process on $E^{\mathbb N}$. By $X_{[s,t]}$, we denote the set of sites visited by the walk between times $[s]$ and $[t]$:
\begin{align*}
X_{[s,t]} = \left\{ X_{[s]}, X_{[s]+1}, \ldots, X_{[t]} \right\}, \textrm{ for } s,t \geq 0.
\end{align*}
We use the notation $e_1, \ldots, e_d$ for the canonical basis of ${\mathbb R}^d$, and denote the segment of length $l \geq 0$ in the $e_i$-direction at $x \in E$ by
$$[x,x+le_i] = E \cap \left\{ x+ \lambda l e_i: \lambda \in [0,1] \right\},$$
where the addition is naturally understood as addition modulo $N$.
The authors of \cite{BS07} introduce a dimension-dependent constant $c_0 > 0$ (cf.~\cite{BS07}, (2.47)) and for any $\beta \in (0,1)$ define an event ${\mathcal G}_{\beta, t}$ for $t \geq 0$ (cf.~\cite{BS07}, (2.52) and Corollary 2.6 in~\cite{BS07}), on which there exists a unique component $O$ of $E \setminus X_{[0,t]}$ containing any segment in $E \setminus X_{[0,t]}$ of the form $\left[x, x+ [c_0 \log N] e_i \right]$, $i=1,\ldots,d$, and such that $O$ is at an $|.|_\infty$-distance of at most $N^\beta$ from any point in $E$. This unique component is referred to as the \emph{giant component}. As in \cite{BS07}, we consider dimensions $d \geq d_0$, with $d_0$ defined as the smallest integer $d_0 \geq 5$ such that 
\begin{align}
49 \left( \frac{2}{d} + \left(1-\frac{2}{d}\right) q(d-2) \right) < 1 \quad \textrm{for any } d \geq d_0, \label{eq:dim}
\end{align}
where $q(d)$ denotes the probability that the simple random walk on ${\mathbb Z}^d$ returns to its starting point. Note that $d_0$ is well-defined, since $q(d) \downarrow 0$ as $d \to \infty$ (see \cite{M56}, (5.4), for precise asymptotics of $q(d)$).
Among other properties of the vacant set, it is shown in \cite{BS07}, Corollary~4.6, that for any dimension $d \geq d_0$ and any $\beta, \gamma \in (0,1)$,
\begin{align}
\lim_N P \left[ {\mathcal G}_{\beta, uN^d} \cap \left\{ \frac{|O|}{N^d} \geq \gamma \right\} \right] = 1, \quad \textrm{for small $u>0$}. \label{eq:Ovol}
\end{align}
Our main result is:
\begin{theorem} \label{thm}
$(d \geq d_0)$
\newline
For any sufficiently small $u>0$, the vacant set left by the random walk on $({\mathbb Z}/N{\mathbb Z})^d$ up to time $uN^d$ contains some segment of length
\begin{align}
l= \left[c_1 \log N \right] \stackrel{\textup{\textrm{(def.)}}}{=} \left[ (300d \log(2d))^{-1} \log N \right], \label{def:l}
\end{align}
which does not belong to the giant component with probability tending to $1$ as $N \to \infty$. That is, for any $\beta \in (0,1)$,
\begin{align}
\lim_N P\left[ {\mathcal G}_{\beta, uN^d} \cap \left( \bigcup_{x \in E} \left\{ [x, x+le_1] \subseteq E \setminus (X_{[0,uN^d]} \cup O) \right\} \right) \right] =1, \textrm{ for small } u>0. \label{eq:thm}
\end{align}
\end{theorem}

We now comment on the strategy of the proof of Theorem~\ref{thm}. We show that for $l$ as in (\ref{def:l}), for some $\nu>0$ and $u>0$ chosen sufficiently small,
\begin{align}
&\textrm{the vacant set at time $\bigl[N^{2-\frac{1}{10}} \bigr]$ contains at least $[N^\nu]$ components consisting of a}\label{claim1} \\
&\textrm{single segment of length $l$ (cf.~Section~\ref{sec:profusion}),} \nonumber\\
&\textrm{with high probability some of these segments remain unvisited until time $[uN^d]$} \label{claim2} \\
&\textrm{(cf.~Section~\ref{sec:thm}).}  \nonumber
\end{align}
Note that these logarithmic components are distinct from the giant component with overwhelming probability in view of (\ref{eq:Ovol}).

\vspace{5pt}
Let us explain the main ideas in the proofs of the claims (\ref{claim1}) and (\ref{claim2}). The argument showing (\ref{claim1}) consists of two steps. The first step is Lemma~\ref{lem:d}, which proves that with high probability, at any two times until $\bigl[N^{2-\frac{1}{10}}\bigr]$ separated by at least $\bigl[N^{\frac{4}{3}}\bigr]$, the random walk is at distinct locations. Here, the fact that $d \geq 5$ plays an important role. 

In the second step, we partition the time interval $\bigl[0,\bigl[N^{2-\frac{1}{10}}\bigr] \bigr]$ into subintervals of length $\bigl[N^{\frac{4}{3}+ \frac{1}{100}} \bigr] > \bigl[N^{\frac{4}{3}}\bigr]$. We show in Lemma~\ref{lem:e} that with high probability, there are at least $[N^\nu]$ such subintervals during which the following phenomenon occurs: the random walk visits every point on the boundary of an unvisited segment of length $l$ without hitting the segment itself, and thereafter also does not visit the segment for a time longer than $\bigl[N^{\frac{4}{3}}\bigr]$. It then follows with the help of the previous Lemma~\ref{lem:d} that the random walk does not visit the surrounded segments at all. Similarly, the segments surrounded in the $[N^\nu]$ different subintervals are seen to be distinct, and claim~(\ref{claim1}) is shown (cf.~Lemma~\ref{lem:geo}). The proof of Lemma~\ref{lem:e} uses a result on the ubiquity of segments of logarithmic size in the vacant set from \cite{BS07}. From this ubiquity result, we know that for any $\beta>0$, with overwhelming probability, there is a segment of length $l$ in the vacant set left until the beginning of every considered subinterval (in fact even until $[uN^d]$ for small $u>0$) in the $N^\beta$-neighborhood of any point. Hence, to show Lemma~\ref{lem:e}, it essentially suffices to find a lower bound on the probability that for some $\beta>0$, the random walk surrounds, but does not visit, a fixed segment in the $N^\beta$-neighborhood of its starting point until time $\bigl[N^{\frac{4}{3} +\frac{1}{100}} /2 \bigr]$ and does not visit the same segment until time $\bigl[N^{\frac{4}{3} +\frac{1}{100}}\bigr] > \bigl[N^{\frac{4}{3} +\frac{1}{100}} /2 \bigr] + \bigl[N^{\frac{4}{3}} \bigr]$.   

\vspace{5pt}
The rough idea behind the proof of claim~(\ref{claim2}) is to use a lower bound on the probability that one fixed segment of length $l$ survives (i.e.~remains unvisited) for a time of at least $[uN^d]$. With estimates on hitting probabilities mentioned in Section~\ref{sec:pre}, it can be shown that this probability is at least $e^{- \textrm{const }ul}$. Since this is much larger than $\frac{1}{[N^\nu]}$ for $u>0$ sufficiently small, cf.~(\ref{def:l}), it should be expected that with high probability, at least one of the $[N^\nu]$ unvisited segments survives until time $[uN^d]$. This conclusion does not follow immediately, because of the dependence between the events that different segments survive. However, the desired conclusion does follow by an application of a technique, developed in \cite{BS07}, for bounding the variance of the total number of segments which survive. 

\vspace{5pt}
The article is organized as follows:

\vspace{5pt}
Section~\ref{sec:pre} contains some estimates on hitting probabilities and exit times recurrently used throughout this work. In Section~\ref{sec:profusion}, we prove claim~(\ref{claim1}). In Section~\ref{sec:survival}, we prove a crucial ingredient for the derivation of claim~(\ref{claim2}). In Section~\ref{sec:thm}, we prove (\ref{claim2}) and conclude that these two ingredients do yield Theorem~\ref{thm}.

\vspace{5pt}
Finally, we use the following convention concerning constants: Throughout the text, $c$ or $c'$ denote positive constants which only depend on the dimension $d$, with values changing from place to place. The numbered constants $c_0, c_1, c_2, c_3, c_4$ are fixed and refer to their first place of appearance in the text. 

\paragraph{Acknowledgments.} The author is grateful to Alain-Sol Sznitman for proposing the problem and for helpful advice.

\section{Some definitions and useful results} \label{sec:pre}

In this section, we introduce some more standard notation and some preliminary estimates on hitting probabilities and exit and return times to be frequently used later on. By $({\mathcal F}_n)_{n \geq 0}$ and $(\theta_n)_{n \geq 0}$ we denote the canonical filtration and shift operators on $E^{\mathbb N}$.
For any set $A \subseteq E$, we often consider the entrance time $H_A$ and the exit time $T_A$, defined as
\begin{align*}
H_A &= \inf \left\{n \geq 0: X_n \in A \right\}, \textrm{ and} \\
T_A &= \inf \left\{n \geq 0: X_n \notin A \right\}.
\end{align*}
For any set $B \subsetneq E$, we denote the Green function of the random walk killed when exiting $B$ as
\begin{align}
g^B(x,y) = E_x \left[ \sum_{n=0}^{\infty} \mathbf{1} \left\{X_n = y, n < T_{B} \right\} \right]. \label{def:g}
\end{align}
We write $|.|_\infty$ for the $l_\infty$-distance on $E$, $B(x,r)$ for the $|.|_\infty$-closed ball of radius $r>0$ centered at $x \in E$, and denote the induced mutual distance of subsets $A$, $B$ of $E$ with $$d(A,B) = \inf \left\{ |x-y|_\infty: x \in A, y \in B \right\}.$$ For any set $A \subseteq E$, the boundary $\partial A$ of $A$ is defined as the set of points in $E \setminus A$ having neighbors in $A$ and the number of points in $A$ is denoted by $|A|$. For sequences $a_N$ and $b_N$, we write $a_N \ll b_N$ to mean that $a_N / b_N$ tends to $0$ as $N$ tends to infinity.

Throughout the proof, we often use the following estimate on hitting probabilities:
\begin{lemma} \label{lem:g}
$(d \geq 1, A \subseteq B \subsetneq E, x \in B)$
\begin{align}
\frac{\sum_{y \in A} g^{B}(x,y)}{\sup \limits_{y \in A} \sum_{y' \in A} g^B (y,y')} \leq P_x \left[ H_A \leq T_{B} \right] \leq \frac{\sum_{y \in A} g^B(x,y)}{\inf \limits_{y \in A} \sum_{y' \in A} g^B (y,y')}. \label{eq:g}
\end{align}
\end{lemma}
\begin{proof}
Apply the strong Markov property at $H_A$ to
$$\sum_{y \in A} g^B(x,y) = E_x \left[ \left\{H_A \leq T_B\right\} , \left(\sum_{y \in A} g^B(X_0,y) \right) \circ \theta_{H_A} \right].$$
\end{proof}
Moreover, we use the following exit-time estimates: 
\begin{lemma} \label{lem:exit}
\emph{($1 \leq a, b < \frac{N}{2}, x \in E$)}
\begin{align}
&P_x \left[T_{B(0,a)} \geq b^2 \right] \leq ce^{-c' \left( \frac{b}{a} \right)^2 }, \label{eq:exit1} \\
&P_0 \left[T_{B(0,b)} \leq a^2 \right] \leq ce^{-c' \frac{b}{a}}. \label{eq:exit2}
\end{align}
\end{lemma}
\begin{proof}
We may assume that $2a \leq b$, for otherwise there is nothing to prove. To show (\ref{eq:exit1}), one uses the Chebychev inequality with $\lambda > 0$ and obtains
\begin{align*}
P_x \left[T_{B(0,a)} \geq b^2 \right] \leq E_x \left[ \exp \left\{ \frac{\lambda}{a^2} T_{B(0,a)} \right\} \right] e^{-\lambda \left( \frac{b}{a} \right)^2}.
\end{align*}
By Kha\'sminskii's Lemma (see \cite{S99}, Lemma 1.1, p.~292, and also \cite{K59}), this last expectation is bounded from above by $2$ for a certain constant $\lambda > 0$, and (\ref{eq:exit1}) follows.
As for (\ref{eq:exit2}), we define the stopping times $(U_n)_{n \geq 1}$ as the times of successive displacements of the walk at distance $a$, i.e.
\begin{align*}
U_1 &= \inf \left\{n \geq 0 : |X_n - X_0|_\infty \geq a \right\}, \textrm{ and for $n \geq 2$, } \\
 U_n &= U_1 \circ \theta_{U_{n-1}} + U_{n-1}. 
\end{align*}
Since $b \geq \left[ \frac{b}{a} \right]a$, one has
$T_{B(0,b)} \geq U_{\left[ \frac{b}{a} \right]}$ $P_0$-a.s.,
hence by the Chebychev inequality and the strong Markov property applied inductively at the times $U_{\left[ \frac{b}{a} \right]-1 }, \ldots, U_{1}$,
\begin{align*}
P_0 \left[T_{B(0,b)} \leq a^2 \right] &\leq e E_0 \left[ \exp \left\{- \frac{1}{a^2} U_{\left[ \frac{b}{a} \right]} \right\} \right] \\
&\stackrel{\textrm{(Markov)}}{\leq} e \left( E_0 \left[ \exp \left\{-\frac{1}{a^2} U_1 \right\} \right] \right)^{\left[ \frac{b}{a} \right]}. 
\end{align*}
By the invariance principle, the last expectation is bounded from above by $1-c$ for some constant $c>0$, from which (\ref{eq:exit2}) follows.
\end{proof}

The following positive constants remain fixed throughout the article,
\begin{align}
\beta_0 = \frac{1}{3(d-2)} \quad < \quad \alpha_0 = \frac{4}{3} \quad < \quad \beta_1 = \frac{4}{3} + \frac{1}{100} \quad < \quad \alpha_1 = 2 - \frac{1}{10}, \label{def:p}
\end{align}
as do the quantities
\begin{align}
b_{0}=[N^{\beta_0}] \quad \ll \quad a_0 =[N^{\alpha_0}] \quad \ll \quad  b_{1} =[N^{\beta_1}] \quad \ll \quad a_1 = [N^{\alpha_1}]. \label{def:ss}
\end{align}
We are now ready to begin the proof of the two crucial claims (\ref{claim1}) and (\ref{claim2}), starting with (\ref{claim1}).

\section{Profusion of logarithmic components until time $a_1$} \label{sec:profusion}

In this section, we show the claim (\ref{claim1}). To this end, we define the ${\mathcal F}_{[t]}$-measurable random subset ${\mathcal J}_t$ of $E$ for $t \geq 0$, as the set of all $x \in E$ such that the segment $[x, x+ le_1]$ forms a component of the vacant set left until time $[t]$, where $l$ was defined in (\ref{def:l}):
\begin{align}
{\mathcal J}_{t} = \left\{ x \in E: X_{[0,t]} \supseteq \partial [x,x+le_1] \textrm{ and } X_{[0,t]} \cap [x,x+le_1] = \emptyset \right\}. \label{def:J}
\end{align}
We then show that for small $\nu>0$, at least $\bigl[ N^{\nu} \bigr]$ segments of length $l$ occur as components in the vacant set until time $a_1$ with overwhelming probability:
\begin{proposition} \label{pr:1}
$(d \geq 5, a_1 \textrm{ as in } (\ref{def:ss}), l \textrm{ as in } \textup{(\ref{def:l})})$
\newline
For small $\nu>0$,
\begin{align}
\lim_N P \left[ |{\mathcal J}_{a_1}| \geq [N^\nu] \right] = 1. \label{eq:pr1}
\end{align}
\end{proposition}

\begin{proof}
The proof of Proposition~\ref{pr:1} will be split into Lemmas~\ref{lem:d}, \ref{lem:e} and \ref{lem:geo}, which we now state. Lemma~\ref{lem:d} asserts that when $d \geq 5$, on an event of probability tending to $1$ as $N$ tends to infinity, $X_I \cap X_J = \emptyset$, for all subintervals $I$, $J$ of $[0,a_1]$ with mutual distance at least $a_0$.
\begin{lemma} \label{lem:d}
$(d \geq 5)$
\newline
\begin{align}
\lim_N P \Biggl[ \bigcap_{n=0}^{a_1-a_0} \left\{X_{[0,n]} \cap  X_{[n+a_0,a_1]} = \emptyset \right\} \Biggr]=1. \label{eq:d}
\end{align}
\end{lemma}
We then consider the $[a_1/b_1]$ subintervals $[(i-1)b_1,ib_1]$, $i=1,\ldots [a_1/b_1]$, of the interval $[0,a_1]$, each of length $b_1$, larger than $a_0$, cf.~(\ref{def:ss}). By ${\mathcal A}_{i,S}$, $S \subseteq E$, we denote the event that, during the first half of the $i$-th time interval, the random walk produces a component consisting of a segment of length $l$ (cf.~(\ref{def:l})) at some point $x \in S$, and does not visit the same component until the end of the $i$-th time interval:
\begin{align}
{\mathcal A}_{i,S} = \bigcup_{x \in S} \bigl( & \left\{ X_{[(i-1)b_1,(i-1)b_1+b_1/2]} \supseteq \partial[x,x+ le_1] \right\} \cap  \label{def:Ai}\\
&\left\{ X_{[0,ib_1]} \cap [x,x+ le_1] = \emptyset \right\} \bigr) \in {\mathcal F}_{ib_1}, \nonumber
\end{align}
for $i=1, \ldots, [a_1/b_1]$. For $S \subseteq E$, the random subset ${\mathcal I}_S$ of $\left\{1, \ldots, [a_1/b_1] \right\}$ is then defined as the set of indices $i$ for which ${\mathcal A}_{i,S}$ occurs, i.e.
\begin{align}
{\mathcal I}_S = \left\{ i \in \{1, \ldots, [a_1/b_1] \}: {\mathcal A}_{i,S} \textrm{ occurs} \right\}. \label{def:I} 
\end{align}
The next lemma then asserts that at least $[N^\nu]$ of the events ${\mathcal A}_{i,E}$, $i=1, \ldots, [a_1/b_1]$, occur.
\begin{lemma} \label{lem:e}
$(d \geq 4)$
\newline
For small $\nu>0$,
\begin{align}
\lim_N P\left[ |\mathcal{I}_E| \geq [N^\nu] \right]=1. \label{eq:e}
\end{align}
\end{lemma}
Finally, Lemma~\ref{lem:geo} shows that Lemmas~\ref{lem:d} and \ref{lem:e} together do yield Proposition~\ref{pr:1}.
\begin{lemma}\label{lem:geo}
$(d \geq 2, \nu>0, N \geq c)$
\begin{align}
\left\{ |{\mathcal I}_E| \geq [N^\nu] \right\} \cap  \bigcap_{n=0}^{a_1-a_0} \left\{X_{[0,n]} \cap  X_{[n+a_0,a_1]} = \emptyset \right\} \quad \subseteq \quad \left\{|{\mathcal J}_{a_1}| \geq [N^\nu] \right\}. \label{eq:geo}
\end{align}
\end{lemma}

We now prove these three Lemmas.
\begin{proof}[Proof of Lemma~\ref{lem:d}.]
We start by observing that by the simple Markov property and translation invariance, the probability of the complement of the event in (\ref{eq:d}) is bounded by
\begin{align}
P \Biggl[ \bigcup_{\substack{n,m \in [0, a_1] \\ m\geq n+a_0}} \left\{X_n = X_m \right\} \Biggr] &\leq \sum_{n=0}^{a_1} P \Biggl[ \bigcup_{m \in [n+a_0,n+a_1]} \left\{X_n = X_m \right\} \Biggr] \label{eq:ld7} \\
&= (a_1+1) P_0 \left[ H_{\{0\}} \circ \theta_{a_0} + a_0 \leq a_1 \right]. \nonumber 
\end{align}
The remaining task is to find an upper bound on this last probability via the exit-time estimates (\ref{eq:exit1}) and (\ref{eq:exit2}). We put $a_* = \bigl[N^{\frac{\alpha_0}{2} -\frac{1}{100}} \bigr] = \bigl[N^{\frac{2}{3} -\frac{1}{100}} \bigr]$. Note that then $a_*^2 \ll a_0$ and $a_1 \ll N^2$. By the exit-time estimates (\ref{eq:exit1}) and (\ref{eq:exit2}), we can therefore assume that the random walk exits the ball $B(0,a_*)$ before time $a_0$, but remains in $B(0,\frac{N}{4})$ until time $a_1$. More precisely, one has
\begin{align}
&P_0 \left[H_{\{0\}} \circ \theta_{a_0} + a_0 \leq a_1 \right] \label{eq:ld1}\\
&\quad \leq P_0 \left[ \left\{ H_{\{0\}} \circ \theta_{a_0} + a_0 \leq a_1 \right\} \cap \left\{ T_{B(0,a_*)} \leq a_0 \right\} \cap \left\{ T_{B(0,\frac{N}{4})} > a_1 \right\} \right] \nonumber\\
&\qquad + P_0 \left[T_{B(0,a_*)} >  a_0 \right] + P_0 \left[T_{B(0,\frac{N}{4})} \leq a_1 \right] \nonumber\\
&\quad = P_1 + P_2 + P_3, \nonumber
\end{align}
where $P_1$, $P_2$ and $P_3$ is abbreviated notation for three terms in the previous line.
By the exit-time estimate (\ref{eq:exit1}) applied with $a=a_*$ and $b=\sqrt{a_0}$, one has
\begin{align}
P_2 = P_0 \left[T_{B(0,a_*)} >  a_0 \right] \leq ce^{-c' \frac{a_0}{a_*^2} } \leq ce^{-c'N^{\frac{1}{50}}}. \label{eq:ld2}
\end{align}
Moreover, the estimate (\ref{eq:exit2}) with $a=\sqrt{a_1}$ and $b=\frac{N}{4}$ implies that
\begin{align}
P_3 = P_0 \left[T_{B(0,\frac{N}{4})} \leq a_1 \right] \leq ce^{-c' \frac{N}{\sqrt{a_1}} } \leq ce^{-c'N^{\frac{1}{20}}}. \label{eq:ld2.1}
\end{align}
It thus remains to bound $P_1$. We obtain by the strong Markov property applied at time $T_{B(0,a_*)}$, that
\begin{align}
P_1 &\leq P_0 \left[ H_{\{0\}} \circ \theta_{T_{B(0,a_*)}} + T_{B(0,a_*)} < T_{B(0,\frac{N}{4})} \right]  \label{eq:ld3}\\
&\stackrel{\textrm{(Markov)}}{\leq} \sup_{x \in E: |x|_\infty = a_* +1} P_x \left[ H_{\{0\}} \leq T_{B(0,\frac{N}{4})} \right] \nonumber
\end{align}
The standard Green function estimate from \cite{lawler}, Theorem~1.5.4.~implies that for any $x \in E$ with $|x|_\infty =a_* +1$, 
\begin{align}
P_x \left[ H_{\{0\}} \leq T_{B \left(0, \frac{N}{4} \right)} \right] \stackrel{(\ref{def:g})}{\leq} g^{B \left(0, \frac{N}{4} \right)} (x,0) \leq ca_*^{-(d-2)} \leq c N^{-(d-2) \left(\frac{\alpha_0}{2} - \frac{1}{100} \right)}. \nonumber
\end{align}
Inserted into (\ref{eq:ld3}), this yields
\begin{align}
P_1 \leq c N^{-(d-2)\left(\frac{\alpha_0}{2} - \frac{1}{100}\right)}. \label{eq:ld4}
\end{align}
Substituting the bounds (\ref{eq:ld2}), (\ref{eq:ld2.1}) and (\ref{eq:ld4}) into (\ref{eq:ld1}), one then finds that
\begin{align*}
P_0 \left[H_{\{0\}} \circ \theta_{a_0} + a_0 \leq a_1 \right] \leq c N^{-(d-2)\left(\frac{\alpha_0}{2} - \frac{1}{100}\right)}. 
\end{align*}
Inserting this estimate into (\ref{eq:ld7}), one finally obtains
\begin{align}
P \Biggl[ \bigcup_{\substack{n,m \in [0, a_1] \\ m\geq n+a_0}} \left\{X_n = X_m \right\} \Biggr] \leq ca_1  N^{-(d-2)\left(\frac{\alpha_0}{2} - \frac{1}{100}\right)} \leq  c N^{\alpha_1 -(d-2)\left(\frac{\alpha_0}{2} - \frac{1}{100}\right)}. \label{eq:ld8}
\end{align}
Since $d-2 \geq 3$, we have
$$\alpha_1 -(d-2)\left(\frac{\alpha_0}{2} - \frac{1}{100}\right) \leq 2- \frac{1}{10} - 3 \left( \frac{2}{3} - \frac{1}{100} \right) = - \frac{7}{100} <0,$$
and the proof of Lemma~\ref{lem:d} is complete with (\ref{eq:ld8}).
\end{proof}

\begin{proof}[Proof of Lemma~\ref{lem:e}.]
The following result on the ubiquity of segments of logarithmic size from \cite{BS07} will be used: Define for any constants $K>0$, $0<\beta<1$ and time $t \geq 0$, the event
\begin{align}
{\mathcal V}_{K, \beta, t} = \Bigl\{ &\textrm{ for all $x \in E$, $1\leq j \leq d$, for some $0 \leq m < N^{\beta}$, } \label{def:V} \\
&X_{[0,t]} \cap \left\{x+ \left(m+\left[0, [K \log N] \right] \right) e_j \right\} = \emptyset \Bigr\}. \nonumber
\end{align}
Then for dimension $d \geq 4$ and some constant $c>0$, one has
\begin{align}
\limsup_{N} \frac{1}{N^c} \log P \left[ {\mathcal V}_{c_1, \beta_0, uN^d}^c \right] <0, \quad \textrm{for small $u>0$}, \label{eq:V}
\end{align}
see the end of the proof of Theorem~1.2 in \cite{BS07} and note the bounds (1.11), (1.49), (1.56) in \cite{BS07}. With this last estimate we will be able to assume that at the beginning of every time interval $[(i-1)b_1, ib_1]$, $i=1, \ldots, [a_1/b_1]$, there is an unvisited segment of length $l$ in the $b_0$-neighborhood of the current position of the random walk. This will reduce the proof of Lemma~\ref{lem:e} to the derivation of a lower bound on $P_0 \left[ {\mathcal A}_{1,\{x\}} \right]$ for an $x$ in the $b_0$-neighborhood of $0$.

We denote with $\mathbb I$ the set of indices, i.e.~${\mathbb I}~=~\left\{1, \ldots, [ a_1/b_1] \right\}.$ A rough counting argument yields the following bound on the probability of the complement of the event in (\ref{eq:e}):
\begin{align}
P \left[ |{\mathcal I}_E|<[N^\nu] \right] &\leq \sum_{\substack{I \subseteq {\mathbb I} \\ |I| \geq |{\mathbb I}| - [N^\nu]}} P \left[ {\mathcal I}_E^c \supseteq I \right] \leq e^{cN^\nu \log N} \sup_{\substack{I \subseteq {\mathbb I} \\ |I| \geq |{\mathbb I}| - N^\nu}} P \left[{\mathcal I}_E^c \supseteq I \right]. \label{eq:le1}
\end{align}
For any set $I$ considered in the last supremum, we label its elements in increasing order as $1 \leq i_1 < \ldots < i_{|I|}$. Note that the events ${\mathcal V}_{c_1,\beta_0,t}$ defined in (\ref{def:V}) decrease with $t$. Applying (\ref{eq:V}), one obtains that
\begin{align}
P \left[ {\mathcal I}_E^c \supseteq I \right] \leq P \left[ \left\{ {\mathcal I}_E^c \supseteq I \right\} \cap {\mathcal V}_{c_1,\beta_0,a_1} \right] + ce^{-N^{c'}}. \label{eq:le2}
\end{align}
Again with monotonicity of ${\mathcal V}_{c_1,\beta_0,t}$ in $t$, one finds
\begin{align}
P \left[ \left\{ {\mathcal I}_E^c \supseteq I \right\} \cap {\mathcal V}_{c_1,\beta_0,a_1} \right] \leq P \left[ \bigcap_{i \in I \setminus \{i_{|I|}\}} {\mathcal A}_{i,E}^c \cap {\mathcal V}_{c_1,\beta_0,(i_{|I|}-1)b_1} \cap {\mathcal A}_{i_{|I|},E}^c \right]. \label{eq:le3}
\end{align}
We now claim that for any event ${\mathcal B} \in {\mathcal F}_{(i-1)b_1}$, $i \in {\mathbb I}$, such that ${\mathcal B} \subseteq {\mathcal V}_{c_1,\beta_0,(i-1)b_1}$, we have 
\begin{align}
P\left[{\mathcal A}_{i,E} \cap {\mathcal B}\right] \geq cb_0^{-(d-2)} N^{-\frac{1}{100}} P \left[ {\mathcal B} \right], \quad \textrm{for } N \geq c'. \label{eq:le4}
\end{align}
Before proving (\ref{eq:le4}), we note that if one uses (\ref{eq:le4}) in (\ref{eq:le3}) with $i=i_{|I|}$ and ${\mathcal B} = \bigcap_{i \in I \setminus \{i_{|I|}\}} {\mathcal A}_{i,E}^c \cap {\mathcal V}_{c_1,\beta_0,(i_{|I|}-1)b_1} \in {\mathcal F}_{(i_{|I|}-1)b_1}$, one obtains for $N \geq c$,
\begin{align*}
P \left[ \bigcap_{i \in I} {\mathcal A}_{i,E}^c \cap {\mathcal V}_{c_1,\beta_0,a_1} \right] \leq P \left[ \bigcap_{i \in I \setminus \{i_{|I|}\}} {\mathcal A}_{i,E}^c \cap {\mathcal V}_{c_1,\beta_0,(i_{|I|}-1)b_1} \right] \left( 1- c'b_0^{-(d-2)} N^{-\frac{1}{100}} \right),
\end{align*}
and proceeding inductively, one has for $0<\nu<(\alpha_1 - \beta_1)/2$ (cf.~(\ref{def:p})) and $N \geq c$,
\begin{align}
P \left[ \left\{ {\mathcal I}_E^c \supseteq I \right\} \cap {\mathcal V}_{c_1,\beta_0,a_1} \right] &\leq \left( 1- c'b_0^{-(d-2)} N^{-\frac{1}{100}} \right)^{|I|} \label{eq:le5}\\
&\leq \exp \left\{-c' N^{-(d-2)\beta_0 - \frac{1}{100} + \alpha_1-\beta_1} \right\} \stackrel{(\ref{def:p})}{\leq} \exp\left\{-c' N^{\frac{1}{6}} \right\}. \nonumber
\end{align}
As a result, (\ref{eq:le1}), (\ref{eq:le2}) and (\ref{eq:le5}) together yield for $0< \nu <(\alpha_1 - \beta_1)/2$ and $N \geq c$,
\begin{align*}
P \left[ |{\mathcal I}_E|<[N^\nu] \right] \leq &\exp \left\{N^{\nu} \log N -c' N^{\frac{1}{6}} \right\} + c'' \exp \left\{ N^\nu \log N - N^{c'} \right\},
\end{align*}
hence (\ref{eq:e}). It therefore only remains to show (\ref{eq:le4}). To this end, we first find a suitable unvisited segment of length $l$ to be surrounded during the $i$-th time interval. We thus define the ${\mathcal F}_{(i-1)b_1}$-measurable random subsets $({\mathcal K}_S)_{S \subseteq E}$ of $E$ of points $x \in S \subseteq E$ such that the segment of length $l$ at site $X_{(i-1)b_1} +x$ is vacant at time $(i-1)b_1$:
\begin{align*}
{\mathcal K}_S = \left\{x \in S: X_{[0,(i-1)b_1]} \cap \left( X_{(i-1)b_1}+x+[0,le_1] \right) = \emptyset \right\}.
\end{align*}
For $N \geq c$, on the event ${\mathcal V}_{c_1, \beta_0, (i-1)b_1}$, for any $y \in E$ there is an integer $0 \leq m \leq b_0$ such that the segment $y + me_1 + [0,le_1]$ is contained in the vacant set left until time $(i-1)b_1$. This implies in particular that with $y = X_{(i-1)b_1}$ (and necessarily $m >0$):
$${\mathcal V}_{c_1, \beta_0,(i-1)b_1} \subseteq \left\{ {\mathcal K}_{[e_1,b_0 e_1]} \neq \emptyset \right\}.$$
Since the event $\mathcal B$ in (\ref{eq:le4}) is a subset of ${\mathcal V}_{c_1, \beta_0, (i-1)b_1}$, it follows that
\begin{align}
P\left[{\mathcal A}_{i,E} \cap {\mathcal B}\right] &= P \left[{\mathcal B} \cap \left\{ {\mathcal K}_{[e_1,b_0 e_1]} \neq \emptyset \right\} \cap {\mathcal A}_{i,E} \right]  \label{eq:le5_0}\\
&= \sum_{\substack{S \subseteq [e_1,b_0 e_1], \\ S \neq \emptyset}} P \left[ {\mathcal B} \cap  \{ {\mathcal K}_{[e_1,b_0 e_1]} = S \} \cap  {\mathcal A}_{i,E}  \right]. \nonumber
\end{align}
Observe that for any $S \subseteq [e_1,b_0 e_1]$, 
$\left\{ {\mathcal K}_{[e_1,b_0 e_1]} =S \right\} \cap \theta^{-1}_{(i-1)b_1} {\mathcal A}_{1,S}  \subseteq {\mathcal A}_{i,S} \subseteq {\mathcal A}_{i,E},$ 
so it follows from (\ref{eq:le5_0}) that
\begin{align}
P\left[{\mathcal A}_{i,E} \cap {\mathcal B}\right] \geq \sum_{\substack{S \subseteq [e_1,b_0 e_1], \\ S \neq \emptyset}} P \left[ {\mathcal B} \cap  \{ {\mathcal K}_{[e_1,b_0 e_1]} = S \} \cap \theta^{-1}_{(i-1)b_1} {\mathcal A}_{1,S} \right]. \nonumber
\end{align}
Note that ${\mathcal K}_{[e_1,b_0 e_1]}$ and $\mathcal B$ are both ${\mathcal F}_{(i-1)b_1}$-measurable. Applying the simple Markov property at time $(i-1)b_1$ to the probability in this last expression and using translation invariance, it follows that
\begin{align}
P\left[{\mathcal A}_{i,E} \cap {\mathcal B}\right]  \geq \inf_{\substack{S \subseteq [e_1,b_0 e_1] \\ S \neq \emptyset}} P_0 \left[ {\mathcal A}_{1,S} \right] P\left[ {\mathcal B} \right] \geq \inf_{x \in [e_1,b_0 e_1]} P_0 \left[ {\mathcal A}_{1,\{x\}} \right] P\left[ {\mathcal B} \right]. \label{eq:le5_1}
\end{align}
In the remainder of this proof, we find a lower bound on $\inf_{x \in [e_1,b_0 e_1]} P_0 \left[ {\mathcal A}_{1,\{x\}} \right]$ in three steps. First, for arbitrary $x \in [e_1,b_0e_1]$, we bound from below the probability that the random walk reaches the boundary $\partial[x,x+le_1]$ within time at most $b_1/4$. Next, we estimate the probability that the random walk, once it has reached $\partial[x,x+le_1]$, covers $\partial[x,x+le_1]$ in $[3dl] \ll b_1/4$ steps. And finally, we find a lower bound on the probability that the random walk starting from $\partial[x,x+le_1]$ does not visit the segment $[x,x+le_1]$ during a time interval of length $b_1$. With this program in mind, note that for $x \in [e_1,b_0 e_1]$ and $N \geq c'$, one has
\begin{align*}
{\mathcal A}_{1,\{x\}} \supseteq  &  \Bigl\{H_{\partial [x,x+l e_1]} \leq \frac{1}{4} b_1 \Bigr\} \cap \left\{ (X \circ \theta_{H_{\partial [x,x+l e_1]}})_{\left[0, [3dl] \right]} = \partial [x,x+le_1] \right\} \\
&\qquad \cap \left\{ (X \circ \theta_{H_{\partial [x,x+le_1]} + [3dl]})_{[0,b_1]} \cap [x,x+le_1] = \emptyset \right\}, \quad P_0 \textrm{-a.s.}
\end{align*}
By the strong Markov property, applied at time $H_{\partial [x,x+le_1]}+[3dl]$, then at time $H_{\partial [x,x+le_1]}$, and translation invariance, one can thus infer that
\begin{align}
&\inf_{x \in [e_1,b_0 e_1]} P_0 \left[ {\mathcal A}_{1,\{x\}} \right] \quad \geq \quad \inf_{x \in E: |x|_\infty \leq b_0} P_0 \biggl[ H_{\partial [x,x+le_1]} \leq \frac{1}{4} b_1 \biggr] \times \label{eq:le6}\\
& \inf_{y \in \partial [0,le_1]} P_y \left[ X_{\left[0, [3dl] \right]} = \partial [0,le_1] \right] \times \inf_{y \in \partial [0,le_1]} P_y \left[ X_{[0,b_1]} \cap [0,le_1] = \emptyset \right] \stackrel{\textrm{(def.)}}{=} L_1 L_2 L_3. \nonumber
\end{align}
We now bound each of the above factors from below. Beginning with $L_1$, we fix $x \in E$ such that $|x|_{\infty} \leq b_0$ and define $b_*=\bigl[N^{\frac{1}{2} \left(\beta_1 - \frac{1}{100} \right)} \bigr] = \bigl[ N^{\frac{2}{3}} \bigr]$ (so that $b_0 \ll b_*$ and $b_*^2 \ll b_1$). We then observe that
\begin{align*}
P_0 \biggl[ H_{\partial [x,x+le_1]} \leq \frac{1}{4} b_1 \biggr] &\geq P_0 \left[ H_{\partial [x,x+le_1]} \leq T_{B(0,b_*)} \right] - P_0 \biggl[ T_{B(0,b_*)} \geq \frac{1}{4} b_1 \biggr]. 
\end{align*}
With (\ref{eq:exit1}), where $a=b_*$ and $b= \sqrt{\frac{b_1}{4}}$, we infer with (\ref{def:p}) that
\begin{align}
P_0 \biggl[ H_{\partial [x,x+le_1]} \leq \frac{1}{4} b_1 \biggr] \geq P_0 \left[ H_{\partial [x,x+le_1]} \leq T_{B(0,b_*)} \right] - c\exp \left\{ -c' N^{\frac{1}{100}} \right\}. \label{eq:le7}
\end{align}
We then use the left-hand estimate of (\ref{eq:g}) to find that
\begin{align*}
P_0 \left[ H_{\partial [x,x+le_1]} \leq T_{B(0,b_*)} \right] \geq \frac{\sum_{y \in \partial[x,x+le_1]} g^{B(0,b_*)}(0,y)}{\sup \limits_{y \in \partial [x,x+le_1]} \sum_{y' \in \partial [x,x+le_1]} g^{B(0,b_*)}(y,y')}. 
\end{align*}
With the Green function estimate of \cite{lawler}, Proposition~1.5.9 (for the numerator) and transience of the simple random walk in dimension $d-1$ (for the denominator), the right-hand side is bounded from below by $clb_0^{-(d-2)}$. With (\ref{eq:le7}), this implies that for $N \geq c$,
\begin{align}
L_1 \geq c'lb_0^{-(d-2)}. \label{eq:le9}
\end{align}
The lower bound we need on $L_2$ in (\ref{eq:le6}) is straightforward: We simply calculate the probability that the random walk follows a suitable fixed path in $\partial [0,le_1]$, starting at $y \in \partial [0,le_1]$ and covering $\partial [0,le_1]$ in at most $d(2l+8) \leq 3dl$ steps (for $N \geq c'$). Such a path can for instance be found by considering the paths ${\mathcal P}_i$, $i=2,\ldots,d$, surrounding the segment $[0,le_1]$ in the $(e_1,e_i)$-hyperplane, i.e.
\begin{align*}
{\mathcal P}_i =& (-1 e_1+0 e_i, - 1e_1+1e_i,0e_1+1e_i, 1e_1+1e_i, \ldots, (l+1)e_1 +1e_i, \\ 
&\quad (l+1)e_1 +0e_i, (l+1)e_1-1e_i, le_1-1e_i, \ldots, -1e_1-1e_i,-1e_1+0e_i),
\end{align*}
$i=2,\ldots,d.$ The paths ${\mathcal P}_i$ visit only points in $\partial [0,le_1]$ and their concatenation forms a path starting at $-e_1$ and covering $\partial [0,le_1]$ in $(d-1)(2l+8)$ steps. Finally, any starting point $y \in \partial [0,le_1]$ is linked to $-e_1$ in $\leq 2l+8$ steps via one of the paths ${\mathcal P}_i$. Therefore, we have
\begin{align}
L_2 \geq \left( \frac{1}{2d} \right)^{3dl} = e^{-(3d \log 2d)l} \stackrel{(\ref{def:l})}{\geq} N^{-\frac{1}{100}}. \label{eq:le10}
\end{align}
For $L_3$ in (\ref{eq:le6}), we note that for any $y \in \partial [0,le_1]$,
\begin{align}
P_y \left[ X_{[0,b_1]} \cap [0,le_1] = \emptyset \right] &\geq P_y \left[ T_{B(0,\frac{N}{4})} < H_{[0,le_1]}, T_{B(0,\frac{N}{4})} > b_1 \right] \label{eq:le11}\\
&\geq  P_y \left[ T_{B(0,\frac{N}{4})} < H_{[0,le_1]} \right] - P_y \left[ T_{B(0,\frac{N}{4})} \leq b_1 \right]. \nonumber
\end{align}
Note that the $d-1$-dimensional projection of $X$ obtained by omitting the first coordinate is a $d-1$-dimensional random walk with a geometric delay of constant parameter. Hence, one finds that for $y \in \partial[0,le_1]$, 
\begin{align}
P_y \left[ T_{B(0,\frac{N}{4})} < H_{[0,le_1]} \right] \geq \frac{d-1}{d} \left(1-q(d-1) \right), \label{eq:tr}
\end{align}
where $q(.)$ is as below (\ref{eq:dim}) and we have used $(d-1)/d$ to bound from below the probability that the projected random walk, if starting from $0$, leaves $0$ in its first step. By translation invariance, for $N \geq c$, the second probability on the the right-hand side of (\ref{eq:le11}) is bounded from above by $P_0 \bigl[ T_{B(0,\frac{N}{8})} \leq b_1 \bigr] \leq \exp \bigl\{-cN^{\frac{1}{3}-\frac{1}{200}} \bigr\}$, with (\ref{eq:exit2}), where $a= \sqrt{b_1}$ and $b= \left[\frac{N}{8} \right]$, cf.~(\ref{def:p}). Hence, we find that
\begin{align}
L_3 \geq c. \label{eq:le12}
\end{align}
Inserting the lower bounds on $L_1$, $L_2$ and $L_3$ from (\ref{eq:le9}), (\ref{eq:le10}) and (\ref{eq:le12}) into (\ref{eq:le6}) and then using (\ref{eq:le5_1}), we have shown (\ref{eq:le4}) and therefore completed the proof of Lemma~\ref{lem:e}.
\end{proof}

\begin{proof}[Proof of Lemma~\ref{lem:geo}.]
We denote the events on the left-hand side of (\ref{eq:geo}) by $A$ and $B$, i.e.
\begin{align}
A= \left\{ |{\mathcal I}_E| \geq [N^\nu] \right\}, \quad B= \bigcap_{n=0}^{a_1-a_0} \left\{X_{[0,n]} \cap  X_{[n+a_0,a_1]} = \emptyset \right\}. \nonumber
\end{align}
We need to show that, if $A \cap B$ occurs, then we can find $[N^\nu]$ segments of length $l$ as components of the vacant set left until time $a_1$. Informally, the reasoning goes as follows: for any of the $[N^\nu]$ events ${\mathcal A}_{i,E}$ occurring on $A$, cf.~(\ref{def:Ai}), the random walk produces in the time interval $(i-1)b_1 + [0,b_1/2]$ a component of the vacant set consisting of a segment of length $l$ and this segment remains unvisited for a further duration of $[b_1/2]$, much larger than $a_0$, cf.~(\ref{def:ss}). However, when $B$ occurs, after a time interval of length $a_0$ has elapsed, the random walk does not revisit any point on the visited boundary of the segment appearing in any of the occurring events ${\mathcal A}_{i,E}$. It follows that the segments appearing in the $[N^\nu]$ different occurring events ${\mathcal A}_{i,E}$ are distinct, unvisited and have a completely visited boundary. 
More precisely, we fix any $N \geq c$ such that 
\begin{align}
a_0 \leq \frac{b_1}{2}, \label{eq:t1}
\end{align}
and assume that the events $A$ and $B$ both occur. We pick $1 \leq i_1 < i_2 < \ldots < i_{[N^\nu]} \leq [a_1/b_1]$ such that the events ${\mathcal A}_{i_j,E}$ occur, and denote one of the segments of the form $[x,x+l e_1]$ appearing in the definition of ${\mathcal A}_{i_j,E}$ by $S_j$, cf.~(\ref{def:Ai}). The proof will be complete once we have shown that 
$$\textrm{$X_{[0,a_1]} \supseteq \partial S_j$, $X_{[0,a_1]} \cap S_j = \emptyset$ and $S_j \neq S_{j'}$ for any $j, j' \in \{1, \ldots, [N^{\nu}] \}$, $j < j'$.}$$ 
That $X_{[0,a_1]} \supseteq \partial S_j$ follows directly from the occurrence of the event ${\mathcal A}_{i_j,E}$ on $A$, cf.~(\ref{def:Ai}). 
To see that $X_{[0,a_1]} \cap S_j = \emptyset$, note first that by definition of ${\mathcal A}_{i_j,E}$,
\begin{align}
X_{[0,i_j b_1]} \cap S_j = \emptyset. \label{eq:t2}
\end{align}
In particular, this implies that $X_{[i_j b_1,a_1]} \nsubseteq S_j$ and that for any $x \in S_j$, there is a point $x' \in \partial S_j$ such that $d\left(x,X_{[i_jb_1,a_1]}\right) \geq d\left(x',X_{[i_jb_1,a_1]}\right)$, hence
\begin{align}
d\left(S_j, X_{[i_jb_1,a_1]} \right) \geq d \left(\partial S_j, X_{[i_jb_1,a_1]}\right). \label{eq:t3}
\end{align}
Moreover, one has on ${\mathcal A}_{i_j,E}$ that $\partial S_j \subseteq X_{[0, i_j b_1 -b_1/2]}$, and by (\ref{eq:t1}), $X_{[0, i_j b_1 -b_1/2]} \subseteq X_{[0, i_j b_1 -a_0]}$. Since $B$ occurs, this yields 
\begin{align}
\partial S_j \cap X_{[i_jb_1,a_1]} = \emptyset, \label{eq:t4}
\end{align}
and hence by (\ref{eq:t3}), $S_j \cap X_{[i_jb_1,a_1]} = \emptyset$. With (\ref{eq:t2}) we deduce that $X_{[0,a_1]} \cap S_j = \emptyset$, as required. Finally, we need to show that $S_j \neq S_{j'}$ for $j<j'$. To this end, note that on ${\mathcal A}_{i_{j'},E}$, $X_{[i_jb_1,a_1]} \supseteq X_{[(i_{j'}-1)b_1,a_1]} \supseteq \partial S_{j'}$, and hence
\begin{align*}
d \left(\partial S_j, \partial S_{j'} \right) \geq d \left(\partial S_j, X_{[i_jb_1,a_1]} \right) \stackrel{(\ref{eq:t4})}{>} 0.
\end{align*}
Hence (\ref{eq:geo}) is proved and the proof of Lemma~\ref{lem:geo} is complete.
\end{proof}

The statement (\ref{eq:pr1}) is now a direct consequence of (\ref{eq:d}), (\ref{eq:e}) and (\ref{eq:geo}), so that the proof of Proposition~\ref{pr:1} is finished. 
\end{proof}

\section{Survival of a logarithmic segment} \label{sec:survival}

This section is devoted to the preparation of the second part of the proof of Theorem~\ref{thm}, that is claim (\ref{claim2}). We show that at least one of the $[N^{\nu}]$ isolated segments produced until time $a_1$ remains unvisited by the walk until time $uN^d$. As mentioned in the introduction, the strategy is to use a lower bound of $e^{-cul}$ on the probability that one fixed segment remains unvisited until a (random) time larger than $uN^d$. The desired statement (\ref{claim2}) would then be an easy consequence if the events $\left\{X_{[0,uN^d]} \cap [x,x+le_1] = \emptyset \right\}$ were independent for different $x \in E$, but this is obviously not the case. However, a technique developed in \cite{BS07} allows to bound the covariance between such events for sufficiently distant points $x$ and $x'$ and with $uN^d$ replaced by the random time $D^x_{l^*(u)}$. Here, $D^x_k$ is defined as the end of the $k$-th excursion in and out of concentric boxes of suitable size centered at $x \in E$, and $l^*(u)$ is chosen such that with high probability, $D^x_{l^*(u)} \geq uN^d$, see (\ref{def:rd}) and (\ref{rd1}) below. The variance bounds from \cite{BS07} and the above-mentioned estimates yield the desired claim in Proposition~\ref{pr:2}. In order to state this proposition, we introduce the integer-valued random variable $\Gamma^J_{[s,t]}$ for $s, t \geq 0$ and $J \subseteq E$, counting the number of sites $x$ in $J$ such that the segment $[x,x+le_1]$ is not visited by $X_{[s,t]}$, i.e.
\begin{align}
\Gamma^J_{[s,t]} = \sum_{x \in J} {\mathbf 1}_{\left\{X_{[s,t]} \cap [x,x+le_1] = \emptyset \right\}}. \label{def:Gamma}
\end{align}
The following proposition asserts that for $\nu >0$ and an arbitrary set $J$ of size at least $[N^\nu]$, when $u>0$ is chosen small enough, $\Gamma^J_{[0,uN^d]}$ is not zero with $P_0$-probability tending to $1$ as $N$ tends to infinity. Combined with the application of the Markov property at time $a_1$, it will play a crucial role in the proof of (\ref{claim2}), cf.~(\ref{eq:thm2}) below.
\begin{proposition} \label{pr:2}
$(d \geq 4, 0 < \nu < 1)$
\newline
For $l$ as in \textup{(\ref{def:l})},
\begin{align}
\lim_N \inf_{\substack{J \subseteq E \\ |J|\geq [N^{\nu}]}} P_0 \left[ \Gamma^J_{[0,uN^d]} \geq 1 \right] =1, \textrm{ for small } u>0. \label{eq:pr2}
\end{align}
\end{proposition}

\begin{proof}
Throughout the proof, we say that a statement applies ``for large $N$'' if the statement applies for all $N$ larger than a constant depending only on $d$ and $\nu$. The central part of the proof is an application of a technique for estimating the covariance of ``local functions'' of distant subsets of points in the torus, developed in \cite{BS07}. In order to apply the corresponding result from \cite{BS07}, we set
\begin{align}
L= \left[(\log N)^2 \right] \label{eq:p2_1}
\end{align}
and, for large $N$, consider any positive integer $r$ such that
\begin{align}
10L \leq r \leq \bigl[ N^{\frac{\nu}{d}} \bigr]. \label{eq:p2_2}
\end{align}
Note that $L$ and $r$ then satisfy (3.1) of \cite{BS07}. We then define the nested boxes
\begin{align}
C(x)=B(x,L) \textrm{, and } {\tilde C}(x) = B (x,r). \label{eq:p2_2_1}
\end{align}
Finally, we consider the stopping times $\left(R^x_k, D^x_k \right)_{k \geq 1}$, the successive returns to $C(x)$ and departures from ${\tilde C}(x)$, defined as in \cite{BS07}, (4.8), by
\begin{align}
R^x_1 &= H_{C(x)}, \quad D^x_1 = T_{{\tilde C}(x)} \circ \theta_{R^x_1} + R^x_1, \textrm{ and for } n \geq 2,\label{def:rd}\\
R^x_n &= R^x_1 \circ \theta_{D^x_{n-1}} + D^x_{n-1}, \quad D^x_n = D^x_1 \circ \theta_{D^x_{n-1}} + D^x_{n-1}, \nonumber
\end{align}
so that $0 \leq R_1 < D_1 < \ldots < R_k < D_k < \ldots$, $P$-a.s. The following estimate from \cite{BS07} on these returns and departures will be used:

\begin{lemma} \label{lem:rd}
$(d \geq 3, L= \left[(\log N)^2 \right], r \geq 10L, N \geq 10r)$ \newline
There is a constant $c_2>0$, such that for $u>0$, $x \in E$,
\begin{align}
&P_0 \left[ R^x_{l^*(u)} \leq uN^d \right] \leq cN^d e^{-c'uL^{d-2}}, \textrm{ with } l^*(u) = \left[c_2uL^{d-2} \right]. \label{rd1} 
\end{align}
\end{lemma}

\begin{proof}[Proof of Lemma \ref{lem:rd}.]
The statement is the same as (4.9) in \cite{BS07}, except that we have here replaced $P$ by $P_0$ and added an extra factor of $N^d$ on the right-hand side of (\ref{rd1}). It therefore suffices to note that
$P \left[ R^x_{l^*(u)} \leq uN^d \right] \geq \frac{1}{N^d} P_0 \left[ R^x_{l^*(u)} \leq uN^d \right].$
\end{proof}

We now control the complement of the event in (\ref{eq:pr2}). To this end, fix any $J \subseteq E$ such that $|J| = \bigl[N^{\nu} \bigr]$ and note that
\begin{align}
P_0 \left[ \Gamma^J_{[0,uN^d]} = 0 \right]  & \leq P_0 \left[ \left\{ \Gamma^J_{[0,uN^d]} = 0 \right\} \cap \left\{ D^x_{l^*(u)} \geq uN^d \textrm{ for all } x \in E \right\} \right] \label{eq:p2_3}\\
&\qquad + P_0 \left[ \textrm{for some } x \in E, R^x_{l^*(u)}<D^x_{l^*(u)}<uN^d \right],  \nonumber\\
& \stackrel{(\ref{rd1})}{\leq} P_0 \left[ {\tilde \Gamma}_{u}=0 \right] + N^c e^{-c'u (\log N)^{2(d-2)}}, \textrm{ where}\nonumber 
\end{align}
\begin{align}
{\tilde \Gamma}_u =  \sum_{x \in J} \mathbf{1}_{\left\{ H_{[x,x+le_1]} > D^x_{l^*(u)} \right\}} \stackrel{\textrm{(def.)}}{=} \sum_{x \in J} h(x), \label{eq:p2_4}
\end{align}
and $l^*(u)$ was defined in (\ref{rd1}). In order to bound the probability in (\ref{eq:p2_3}), we need an estimate on the variance of ${\tilde \Gamma}_u$. This estimate can be obtained by using the bound on the covariance of $h(x)$ and $h(y)$ for $x$ and $y$ sufficiently far apart, derived in \cite{BS07}. To this end, one first notes that
\begin{align*}
\textrm{var}_{P_0}\left({\tilde \Gamma}_u \right) &= \textrm{var}_{P_0} \left( \sum_{x \in J} h(x) \right) \leq c \left( N^\nu r^d + r^{2d} \right) + N^{2\nu} \sup_{\substack{x,y \in E \\ |x-y|_\infty \geq 2r+3 \\ x,y \notin {\tilde C}(0)}} \textrm{cov}_{P_0} \left(h(x), h(y) \right).
\end{align*}
In the proof of Proposition~4.2 in \cite{BS07}, the covariance in the last supremum is bounded from above by $cu \frac{L^d}{r}$ (cf.~\cite{BS07}, above (4.44)). Since $r^d \leq N^\nu$ (cf.~(\ref{eq:p2_2})), we therefore have
\begin{align}
\textrm{var}_{P_0} \left({\tilde \Gamma}_u \right) \leq c \left( r^{d} N^\nu + u \frac{N^{2\nu} L^d}{r} \right). \label{eq:p2_5}
\end{align}
Below, we will show that
\begin{align}
P_0 \left[ H_{[x,x+le_1]} > D^x_{l^*(u)} \right] \geq ce^{-c'ul}, \quad \textrm{when } 0 \notin [x,x+le_1]. \label{eq:p2_10}
\end{align}
Before we prove this claim, we show how to deduce Proposition~\ref{pr:2} from the above. It follows from (\ref{eq:p2_10}) that for large $N$,
\begin{align*}
E_0 \left[ {\tilde \Gamma}_u  \right] = \sum_{x \in J} P_0 \left[ H_{[x,x+le_1]} > D^x_{l^*(u)} \right] \geq c_3 N^\nu e^{-c_4 ul}.
\end{align*}
Hence for large $N$, one has
\begin{align}
P_0 \left[ {\tilde \Gamma}_u =0 \right] &\leq P_0 \left[ {\tilde \Gamma}_u < E_0 \left[ {\tilde \Gamma}_{u} \right] - \frac{c_3}{2} N^\nu e^{-c_4ul  } \right] \label{eq:p2_10_1}\\
&\leq c \textrm{ var}_{P_0} ({\tilde \Gamma}_u) N^{-2\nu} e^{cul } \stackrel{(\ref{eq:p2_5})}{\leq} c \left( \frac{r^{d}}{N^{\nu}} + u \frac{L^d}{r} \right) e^{cul }. \nonumber
\end{align}
We now choose $r= \left[ \left(L^d N^{\nu} \right)^{\frac{1}{d+1}} \right]$, so that with (\ref{eq:p2_1}) one has $$cr \leq (\log N)^{\frac{2d}{d+1}} N^{\frac{\nu}{d+1}} \leq c'r$$ and $r$ satisfies (\ref{eq:p2_2}) for large $N$. Inserting these choices of $r$, $L$ and $l$ from (\ref{def:l}) into the estimate (\ref{eq:p2_10_1}), one obtains
\begin{align*}
P_0 \left[ {\tilde \Gamma}_u =0 \right] \leq c(1+u) (\log N)^{c} N^{- \frac{\nu}{d+1} + cu}.
\end{align*}
For $u>0$ chosen sufficiently small, the right-hand side tends to $0$ as $N \to \infty$. With (\ref{eq:p2_3}) and monotonicity of $\Gamma^J_.$ in $J$, this proves (\ref{eq:pr2}). There only remains to show (\ref{eq:p2_10}). 

First, the strong Markov property applied at time $T_{C(x)}$ yields that 
\begin{align}
P_0 \left[H_{[x,x+le_1]} > D^x_{l^*(u)} \right] &\geq P_0 \left[H_{[x,x+le_1]} > D^x_{l^*(u)}, T_{C(x)} < H_{[x,x+le_1]} \right] \nonumber\\
&\geq P_0 \left[T_{C(x)} < H_{[x,x+le_1]} \right] \inf_{y \notin C(x)} P_y \left[H_{[x,x+le_1]} > D^x_{l^*(u)} \right]. \nonumber
\end{align}
For $x$ such that $0 \notin [x,x+le_1]$, transience of simple random walk in dimension $d-1$ implies that $P_0 \left[T_{C(x)} < H_{[x,x+le_1]} \right] \geq c > 0$, and hence,
\begin{align}
P_0 \left[H_{[x,x+le_1]} > D^x_{l^*(u)} \right] &\geq c \inf_{y \notin C(x)} P_y \left[H_{[x,x+le_1]} > D^x_{l^*(u)} \right]. \label{eq:p2_12}
\end{align}
The application of the strong Markov property at the times $R^x_{l^*(u)}, R^x_{l^*(u)-1}, \ldots, R^x_{1}$ then yields 
\begin{align}
\inf_{y \notin C(x)} P_y \left[H_{[x,x+le_1]} > D^x_{l^*(u)} \right] \geq \left( \inf_{y \in \partial (C(x)^c)} P_y \left[H_{[x,x+le_1]} > D^x_{1} \right] \right)^{l^*(u)}. \label{eq:p2_13}
\end{align}
From the right-hand estimate of (\ref{eq:g}) on the hitting probability with $A=[x,x+le_1]$ and $B={\tilde C}(x)$ and the trivial lower bound of $1$ for the denominator of the right-hand side, one obtains that
\begin{align*}
\sup_{y \in \partial (C(x)^c)} P_y \left[H_{[x,x+le_1]} \leq D^x_{1} \right] \leq \sup_{y \in \partial (C(x)^c)} \sum_{z \in [x,x+le_1]} g^{{\tilde C}(x)}(y,z) \leq clL^{-(d-2)}, 
\end{align*}
with the Green function estimate from \cite{lawler}, Theorem~1.5.4 in the last step. Inserting this bound into (\ref{eq:p2_13}), one deduces that
\begin{align*}
\inf_{y \notin C(x)} P_y \left[H_{[x,x+le_1]} > D^x_{l^*(u)} \right] \geq \left(1- clL^{-(d-2)} \right)^{l^*(u)} \geq e^{-c'l L^{-(d-2)}l^*(u)} \geq e^{-c''ul}.
\end{align*}
With (\ref{eq:p2_12}), this shows (\ref{eq:p2_10}) and thus completes the proof of Proposition~\ref{pr:2}.
\end{proof}

\section{Proof of the main result} \label{sec:thm}

Finally, we combine the results of the two previous sections to deduce Theorem~\ref{thm} as a corollary of Propositions~\ref{pr:1} and \ref{pr:2}. 

\begin{proof}[Proof of Theorem~\ref{thm}.]
Note that if the giant component $O$ has macroscopic volume, then any component consisting only of a segment of length $l$ must be distinct from $O$. In other words, one has for $N \geq c$, cf.~(\ref{def:J}),
\begin{align*}
&{\mathcal G}_{\beta, uN^d} \cap \bigcup_{x \in E} \left\{ [x, x+le_1] \subseteq E \setminus (X_{[0,uN^d]} \cup O) \right\} \supseteq {\mathcal G}_{\beta, uN^d} \cap \left\{ \frac{|O|}{N^d} \geq \frac{1}{2} \right\} \cap \left\{ {\mathcal J}_{uN^d} \neq \emptyset \right\}.
\end{align*}
In view of (\ref{eq:Ovol}), it hence suffices to show that
\begin{align}
\lim_N P \left[ {\mathcal J}_{uN^d} \neq \emptyset \right] =1, \textrm{ for small } u>0. \label{eq:thm1}
\end{align}
However, the event in (\ref{eq:thm1}) occurs as soon as there are at least $[N^\nu]$, $\nu>0$, segments of length $l$ as components in the vacant set at time $a_1$, at least one of which is not visited by the random walk until time $uN^d$. For any $\nu>0$ and large $N$ (depending on $\nu$), the probability in (\ref{eq:thm1}) is therefore bounded from below by (cf.~(\ref{def:Gamma}))
\begin{align*}
P \left[ \left\{ |{\mathcal J}_{a_1}|  \geq [N^\nu] \right\} \cap \left\{ \Gamma^{{\mathcal J}_{a_1}}_{[a_1,uN^d]} \geq 1 \right\} \right] = \sum_{\substack{J \subseteq E \\ |J| \geq [N^{\nu}]}} P \left[ \left\{ {\mathcal J}_{a_1} = J \right\}  \cap \left\{ \Gamma^{J}_{[a_1,uN^d]} \geq 1 \right\} \right].
\end{align*}
By the simple Markov property applied at time $a_1$ and translation invariance, one deduces that
\begin{align}
P \left[ {\mathcal J}_{uN^d} \neq \emptyset \right] &\geq \sum_{\substack{J \subseteq E \\ |J| \geq [N^{\nu}]}} P \left[ {\mathcal J}_{a_1} = J \right]  \inf_{\substack{J' \subseteq E \\ |J'| \geq [N^{\nu}]}} P_0 \left[ \Gamma^{J'}_{[0,uN^d]} \geq 1 \right] \label{eq:thm2}\\
& = P \left[ |{\mathcal J}_{a_1}| \geq [N^\nu] \right] \inf_{\substack{J \subseteq E \\ |J| \geq [N^{\nu}]}} P_0 \left[ \Gamma^J_{[0,uN^d]} \geq 1 \right]. \nonumber
\end{align}
For small $\nu>0$, this last quantity tends to $1$ as $N \to \infty$ if $u>0$ is chosen small enough, by (\ref{eq:pr1}) and (\ref{eq:pr2}). This completes the proof of (\ref{eq:thm1}) and hence of Theorem~\ref{thm}.
\end{proof}

\begin{remark}
$ $\newline
\textup{1) With only minor modifications, the proof presented in this work shows that for $u>0$ chosen sufficiently small, on an event of probability tending to $1$ as $N$ tends to infinity, the vacant set left until time $[uN^d]$ contains at least $[N^{c(u)}]$ segments of length $l$, for a constant $c(u)$ depending on $d$ and $u$. Indeed, the proof of Proposition~\ref{pr:2}, with obvious changes, shows that for an arbitrary set $J \subseteq E$ of size at least $[N^\nu]$, one has $\Gamma^J_{[0,uN^d]} \geq \frac{c_3}{2} N^\nu e^{-c_4 ul}$ with probability tending to $1$ as $N$ tends to infinity, if $u>0$ is chosen sufficiently small, and this result can be used in the above proof to show the claim just made.}
\vspace{12pt}

\noindent \textup{2) From results of \cite{BS07} and the present work, it follows that uniqueness of a connected component of $E \setminus X_{[0,uN^d]}$ containing segments of length $[c\log N]$ holds for a certain $c=c_0$ (cf.~(0.7) in \cite{BS07}) and fails for a certain $c=c_1$ with overwhelming probability, when $u>0$ is chosen sufficiently small. It is thus natural to consider the value
\begin{align*}
c_* &= \inf \{c>0: \textrm{for small } u>0, \lim_N P[{\mathcal O}_{c,u}] = 1\}, \textrm{ where} \\
{\mathcal O}_{c,u} &\stackrel{\textrm{(def.)}}{=} \bigl\{ E \setminus X_{[0,uN^d]} \textrm{ contains exactly one connected component} \\
&\qquad \qquad \textrm{containing segments of length $[c \log N]$} \bigr\}.
\end{align*}
The results in \cite{BS07} show in particular that $c_* < \infty$, and the present work shows that $c_* >0$, hence $c_*$ is non-degenerate for $d \geq d_0$. One may then ask if it is true that for arbitrary $0<c<c_*<c'$, $\lim_N P[{\mathcal O}_{c,u}] = 0$ and $\lim_N P[{\mathcal O}_{c',u}]=1$, when $u>0$ is chosen sufficiently small. In fact, using results from \cite{BS07}, one easily deals with the case $c'>c_*$. Indeed, on the event ${\mathcal V}_{c',1/2,uN^d}$ (defined in (\ref{def:V})), the events ${\mathcal O}_{c'',u}$ increase in $c'' \leq c'$, so that one has ${\mathcal V}_{c',1/2,uN^d} \cap {\mathcal O}_{c'',u} \subseteq {\mathcal O}_{c',u}$ for $c'' \leq c'$. Since $\lim_N P[{\mathcal V}_{c',1/2,uN^d}] =1$ for $u>0$ chosen small enough (cf.~(1.26) in \cite{BS07}), this implies that if $\lim_N P[{\mathcal O}_{c'',u}]=1$, then $\lim_N P[{\mathcal O}_{c',u}]=1$ for any $c'>c''$. As far as the value or the large-$d$-behavior of $c_*$ is concerned, only little follows from \cite{BS07} and this work. While the upper bound from \cite{BS07} (cf.~(2.47) in \cite{BS07}) behaves like $d (\log d)^{-1}$ for large $d$, our lower bound behaves like $(d \log d)^{-1}$ (cf.~(\ref{def:l})), which leaves much scope for improvement.
}

\vspace{12pt}
\noindent \textup{3) This work shows a lower bound on non-giant components of the vacant set. Apart from the fact that vacant segments outside the giant component cannot be longer than $[c_0 \log N]$, little is known about upper bounds on such components. Although (\ref{eq:Ovol}) does imply that the volume of a non-giant component of the vacant set is with overwhelming probability not larger than $(1-\gamma)N^d$ for arbitrary $\gamma \in (0,1)$, when $u>0$ is small enough, simulations indicate that the volume of such components is typically much smaller. Further related open questions are raised in \cite{BS07}.}
\end{remark}

\end{document}